\documentclass[reqno]{amsart}%
\usepackage{amsmath}
\usepackage{graphicx}
\usepackage{array}
\usepackage{epsfig}
\usepackage{float,array}
\usepackage{amsfonts}
\usepackage{amssymb}
\usepackage[all,cmtip]{xy}
\usepackage{xcolor}
\usepackage[colorlinks=true]{hyperref}
\usepackage{setspace}
\usepackage{listings}
\usepackage{color}%
\providecommand{\U}[1]{\protect\rule{.1in}{.1in}}
\pdfoutput=1
\hypersetup{linkcolor={red!80!black},
			citecolor={blue!80!black},
			urlcolor={blue!80!black}}
\newtheorem{theorem}{Theorem}[section]
\newtheorem{corollary}[theorem]{Corollary}

\theoremstyle{definition}

\newtheorem{examples}[theorem]{Examples}

\newtheorem{note}[theorem]{Note}
\numberwithin{equation}{section}
\definecolor{mygreen}{rgb}{0,0.6,0}
\definecolor{mygray}{rgb}{0.5,0.5,0.5}
\definecolor{mymauve}{rgb}{0.58,0,0.82}
\definecolor{lightgray}{rgb}{0.95,0.95,0.95}
\begin{document}
\author[D.I. Dais and I. Markakis]{Dimitrios I. Dais and Ioannis Markakis}
\address{University of Crete, Department of Mathematics and Applied Mathematics,
Division Algebra and Geometry, Voutes Campus, P.O. Box 2208, GR-70013,
Heraklion, Crete, Greece}
\email{ddais@math.uoc.gr/johnmarkakis95@gmail.com}
\title[minimal generating systems for some special toric ideals]{Computing minimal generating systems \\for some special toric ideals}
\date{}

\begin{abstract}
Let $X_{P}$ be the projective toric surface associated to a lattice polygon
$P$. If the number of lattice points lying on the boundary of $P$ is at least
$4$, it is known that $X_{P}$ is embeddable into a suitable projective space
as zero set of finitely many quadrics. In this case, the determination of a
minimal generating system of the toric ideal defining $X_{P}$ is reduced to a
simple Gaussian elimination.

\end{abstract}
\maketitle

\section{Introduction\label{INTRO}}

\noindent Let $P\subset\mathbb{R}^{2}$ be a \textit{lattice polygon,} i.e., a
(convex, 2-dimensional) polygon, all of whose vertices belong to
$\mathbb{Z}^{2}.$ $P$ is known to be \textit{normal} and \textit{very ample},
and to have a canonical presentation%
\[
P=\left\{  \mathbf{x}\in\mathbb{R}^{2}\left\vert \left\langle \mathbf{x}%
,\mathbf{u}_{F}\right\rangle \geq-a_{F}\text{ for all }F\in\mathcal{F}%
(P)\right.  \right\}  ,
\]
where $\mathcal{F}(P)$ is the set of the facets (edges) of $P,$ $a_{F}%
\in\mathbb{Z}$ and $\mathbf{u}_{F}\in\mathbb{Z}^{2}$ the inward-pointing facet
normal, i.e., the minimal generator of the ray $\mathbb{R}_{\geq0}%
\mathbf{u}_{F}.$ The corresponding compact complex toric surface $X_{P}$ is
therefore normal and projective, and
\[
D_{P}:= {\displaystyle\sum_{F\in\mathcal{F}(P)}} a_{F}\,\overline
{\text{orb}(\mathbb{R}_{\geq0}\mathbf{u}_{F})}%
\]
is a very ample Cartier divisor on $X_{P},$ where $\overline{\text{orb}%
(\mathbb{R}_{\geq0}\mathbf{u}_{F})}$ is the Zariski closure of the orbit of
the ray $\mathbb{R}_{\geq0}\mathbf{u}_{F}$ w.r.t. the natural Hom$_{\mathbb{Z}%
}(\mathbb{Z}^{2},\mathbb{C}^{\ast})$-action. (See \cite[Corollaries 2.2.13 and
2.2.19 (b), pp. 70-71, (4.2.6), p. 182, and Proposition 6.1.10 (c), p. 269
]{CLS}.) Setting $\delta_{P}:=\sharp(P\cap\mathbb{Z}^{2})-1,$ the complete
linear system $\left\vert D_{P}\right\vert $ induces the closed embedding
$\Phi_{\left\vert D_{P}\right\vert }$,%

\[
\xymatrix{\mathbb{T} \hspace{0.2cm} \ar@{^{(}->}[rr]^{\iota} \ar@/_1.7pc/[rrrr] & & X_{P} \hspace{0.2cm} \ar@{^{(}->}[rr]^{\Phi_{\left\vert
			D_{P}\right\vert }} & & \mathbb{P}_{\mathbb{C}}^{\delta_{P}}}
\]
with
\[
\mathbb{T}\ni t\longmapsto(\Phi_{\left\vert D_{P}\right\vert }\circ
\iota)(t):=[...:z_{(i,j)}:...]_{(i,j)\in P\cap\mathbb{Z}^{2}}\in
\mathbb{P}_{\mathbb{C}}^{\delta_{P}},\ z_{(i,j)}:=\chi^{(i,j)}(t),
\]
where $\chi^{(i,j)}:\mathbb{T}\rightarrow\mathbb{C}^{\ast}$ is the character
associated to the lattice point $(i,j)$ (with $\mathbb{T}$ denoting the
algebraic torus Hom$_{\mathbb{Z}}(\mathbb{Z}^{2},\mathbb{C}^{\ast})$), for all
$(i,j)\in P\cap\mathbb{Z}^{2}.$ The image $\Phi_{\left\vert D_{P}\right\vert
}(X_{P})$ of $X_{P}$ under $\Phi_{\left\vert D_{P}\right\vert }$ is the
Zariski closure of Im$(\Phi_{\left\vert D_{P}\right\vert }\circ\iota)$ in
$\mathbb{P}_{\mathbb{C}}^{\delta_{P}}$ and can be viewed as the projective
variety Proj$(S_{P}),$ where
\[
S_{P}:=\mathbb{C}[C(P)\cap\mathbb{Z}^{3}]=%
{\displaystyle\bigoplus\limits_{\kappa=0}^{\infty}}
\left(
{\displaystyle\bigoplus\limits_{(i,j)\in(\kappa P)\cap\mathbb{Z}^{2}}}
\mathbb{C\cdot}\chi^{(i,j)}s^{\kappa}\right)  \
\]
$($with $C(P):=\{(\lambda y_{1},\lambda y_{2},\lambda)\left\vert \lambda
\in\mathbb{R}_{\geq0}\text{ and }(y_{1},y_{2})\in P\right.  \})$ is the
semigroup algebra which is naturally graded by setting deg$(\chi
^{(i,j)}s^{\kappa}):=\kappa.$ (For a detailed exposition see \cite[Theorem
2.3.1, p. 75; Proposition 5.4.7, pp. 237-238; Theorem 5.4.8, pp. 239-240, and
Theorem 7.1.13, pp. 325-326]{CLS}.) Equivalently, it can be viewed as the zero
set $\mathbb{V}(I_{\mathcal{A}_{P}})\subset\mathbb{P}_{\mathbb{C}}^{\delta
_{P}}$ of the homogeneous ideal $I_{\mathcal{A}_{P}}:=$ Ker$(\pi_{P}),$ where%
\[
\mathcal{A}_{P}:=\left\{  (i,j,1)\left\vert (i,j)\in P\cap\mathbb{Z}%
^{2}\right.  \right\}  \subset\mathbb{Z}^{2}\times\{1\}\subset\mathbb{Z}^{3},
\]
and $\pi_{P}$ is the $\mathbb{C}$-algebra homomorphism
\[
\mathbb{C}[...:z_{(i,j)}:....]_{(i,j)\in P\cap\mathbb{Z}^{2}}\overset{\pi_{P}%
}{\longrightarrow}\mathbb{C}[...,\chi^{(i,j,1)},....]_{(i,j,1)\in
\mathcal{A}_{P}},\ \ z_{(i,j)}\longmapsto\chi^{(i,j,1)}.
\]

\begin{theorem}
[Koelman \cite{Koelman}]\label{KOELMANTHM}If $\sharp(\partial P\cap
\mathbb{Z}^{2})\geq4,$ then $I_{\mathcal{A}_{P}}$ is generated by all possible
quadratic binomials, i.e.,{\small
\[
I_{\mathcal{A}_{P}}=\left\langle \left\{  z_{(i_{1},j_{1})}z_{(i_{2},j_{2}%
)}-z_{(i_{1}^{\prime},j_{1}^{\prime})}z_{(i_{2}^{\prime},j_{2}^{\prime}%
)}\left\vert
\begin{array}
[c]{c}%
(i_{1},j_{1}),(i_{2},j_{2}),(i_{1}^{\prime},j_{1}^{\prime}),(i_{2}^{\prime
},j_{2}^{\prime})\in P\cap\mathbb{Z}^{2},\medskip\\
\text{\emph{with} }(i_{1},j_{1})+(i_{2},j_{2})=(i_{1}^{\prime},j_{1}^{\prime
})+(i_{2}^{\prime},j_{2}^{\prime})
\end{array}
\right.  \right\}  \right\rangle .
\]
}
\end{theorem}

\begin{corollary}
[{Castryck \& Cools \cite[\S 2]{CC}}]\label{CCCOR copy(1)}If $\sharp(\partial
P\cap\mathbb{Z}^{2})\geq4,$ and if we denote by $\beta_{P}$ the cardinality of
any minimal system of quadrics generating the ideal $I_{\mathcal{A}_{P}},$
then%
\begin{equation}
\beta_{P}=\tbinom{\delta_{P}+2}{2}-\sharp(2P\cap\mathbb{Z}^{2}).
\label{BETAQFORMULA}%
\end{equation}

\end{corollary}

\begin{proof}
If HP$_{2}(\mathbb{P}_{\mathbb{C}}^{\delta_{P}}):=\left\{  \text{homogeneous
polynomials (in }\delta_{P}+1\text{ variables) of degree }2\right\}  ,$ then
the $\mathbb{C}$-vector space homomorphism
\[
f:\text{HP}_{2}(\mathbb{P}_{\mathbb{C}}^{\delta_{P}})\longrightarrow
\mathbb{C}\left[  x^{\pm1},y^{\pm1}\right]  ,\text{ mapping }z_{(i_{1},j_{1}%
)}z_{(i_{2},j_{2})}\text{ onto }x^{i_{1}+i_{2}}y^{j_{1}+j_{2}},
\]
has as kernel Ker$(f)$ the $\mathbb{C}$-vector space of homogeneous
polynomials of degree $2$ which belong to $I_{\mathcal{A}_{P}}$ and as image
Im$(f)$ the linear span of $\left.  \{x^{i}y^{j}\right\vert (i,j)\in
2P\cap\mathbb{Z}^{2}\}$ (because every lattice point in $2P$ is the sum of two
lattice points of $P,$ cf. \cite[Theorem 2.2.12, pp. 68-69]{CLS}). Taking into
account Koelman's Theorem \ref{KOELMANTHM}, \cite[Lemma 4.1, p. 31]%
{Sturmfels}, and the fact that $\mathbb{V}(I_{\mathcal{A}_{P}})$ is not
contained in any hyperplane of $\mathbb{P}_{\mathbb{C}}^{\delta_{P}},$ the
equality $\dim_{\mathbb{C}}(\text{Ker}(f))=\dim_{\mathbb{C}}(\text{HP}%
_{2}(\mathbb{P}_{\mathbb{C}}^{\delta_{P}}))-\dim_{\mathbb{C}}%
(\operatorname{Im}(f))$ gives (\ref{BETAQFORMULA}).
\end{proof}

\begin{examples}
\label{Examplintro}(i) If $a,b$ are two positive integers, then the projective
toric surface $X_{P_{a,b}}\cong\mathbb{V}(I_{\mathcal{A}_{P_{a,b}}}%
)\subset\mathbb{P}_{\mathbb{C}}^{\delta_{P}}$ which is associated to the
lattice quadrilateral
\[
P_{a,b}:=\text{conv}(\{(0,0),(a,0),(b,1),(0,1)\})
\]
(where \textquotedblleft conv\textquotedblright\ stands for \textit{convex
hull}, and $\delta_{P_{a,b}}=a+b+1$), is isomorphic to the intersection of
\begin{align*}
\beta_{P_{a,b}}  &  =\tbinom{\delta_{P_{a,b}}+2}{2}-\sharp(2P_{a,b}%
\cap\mathbb{Z}^{2})\smallskip\\
&  =\tfrac{(a+b+2)(a+b+3)}{2}-3(a+b+1)=\allowbreak\tfrac{1}{2}\left(
a+b-1\right)  \left(  a+b\right)
\end{align*}
quadrics, i.e., to the \textit{rational normal scroll of type} $(a,b)$ w.r.t.
the homogeneous coordinates $[...:z_{(i,j)}:....]_{(i,j)\in P_{a,b}%
\cap\mathbb{Z}^{2}}$ satisfying the \textquotedblleft2-minors
condition\textquotedblright\
\[
\text{rank}\left(
\begin{array}
[c]{cccccccc}%
z_{(0,0)} & z_{(1,0)} & \cdots & z_{(a-1,0)} & z_{(0,1)} & z_{(1,1)} & \cdots
& z_{(b-1,1)}\\
z_{(1,0)} & z_{(2,0)} & \cdots & z_{(a,0)} & z_{(1,1)} & z_{(2,1)} & \cdots &
z_{(b,1)}%
\end{array}
\right)  \leq1.
\]
In particular, for $a=b=1,$ $X_{P_{1,1}}\cong\mathbb{V}(z_{(0,0)}%
z_{(1,1)}-z_{(1,0)}z_{(0,1)})\subset\mathbb{P}_{\mathbb{C}}^{3}$ can be viewed
as the classical (smooth) \textit{quadric hypersurface in} $\mathbb{P}%
_{\mathbb{C}}^{3}$ (which is isomorphic to $\mathbb{P}_{\mathbb{C}}^{1}%
\times\mathbb{P}_{\mathbb{C}}^{1}$ and birationally equivalent to
$\mathbb{P}_{\mathbb{C}}^{2},$ cf. Figure \ref{Fig.1}).

\begin{figure}[th]
\includegraphics[height=4cm,width=11cm]{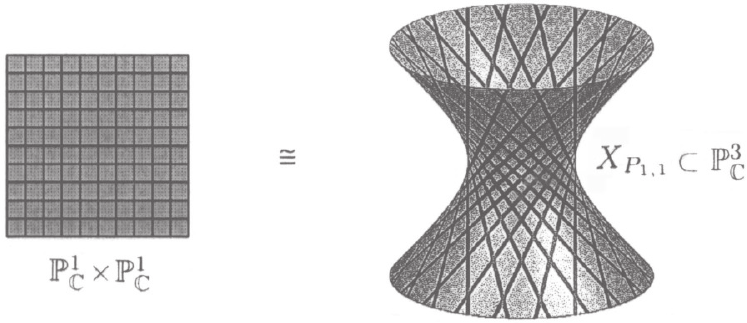} \caption{} \label{Fig.1}%
\end{figure}

\noindent(ii) Let $d$ be a positive integer. The $\mathbb{C}$-vector space
\[
\mathbb{C}[\mathsf{X}_{0},\mathsf{X}_{1},\mathsf{X}_{2}]_{d}:=\left\{
F\in\mathbb{C}[\mathsf{X}_{0},\mathsf{X}_{1},\mathsf{X}_{2}]\left\vert F\text{
homogeneous of degree }d\right.  \right\}  \cup\{\mathbf{0}\}
\]
has the set $\{\mathsf{X}_{0}^{\alpha_{0}}\mathsf{X}_{1}^{\alpha_{1}%
}\mathsf{X}_{2}^{\alpha_{2}}\left\vert (\alpha_{0},\alpha_{1},\alpha_{2}%
)\in\mathcal{E}_{2,d}\right.  \}$ as one of its bases, where%
\[
\mathcal{E}_{2,d}:=\left\{  \boldsymbol{\alpha}=(\alpha_{0},\alpha_{1}%
,\alpha_{2})\in\mathbb{Z}^{3}\left\vert \alpha_{0},\alpha_{1},\alpha_{2}%
\in\lbrack0,d]\text{ and }\alpha_{0}+\alpha_{1}+\alpha_{2}=d\right.  \right\}
.
\]
For each $\boldsymbol{\alpha}=(\alpha_{0},\alpha_{1},\alpha_{2})\in
\mathcal{E}_{2,d}$ we write $\mathsf{X}^{\boldsymbol{\alpha}}:=\mathsf{X}%
_{0}^{\alpha_{0}}\mathsf{X}_{1}^{\alpha_{1}}\mathsf{X}_{2}^{\alpha_{2}}.$
Setting
\[
\text{Tr}_{d}:=\text{conv}(\{(0,0),(d,0),(0,d)\})=\left\{  (x_{1},x_{2}%
)\in\mathbb{R}^{2}\left\vert x_{1}\geq0,x_{2}\geq0\text{ and }x_{1}+x_{2}\leq
d\right.  \right\}
\]
we see that $\sharp(\partial$Tr$_{d}\cap\mathbb{Z}^{2})=3d$ and $\sharp
($Tr$_{d}\cap\mathbb{Z}^{2})=\tbinom{d+2}{2},$ because%
\[
\text{Tr}_{d}\cap\mathbb{Z}^{2}\ni(m_{1},m_{2})\longmapsto(m_{1}+1,m_{1}%
+m_{2}+2)\in\left\{  (\xi_{1},\xi_{2})\in\mathbb{Z}_{\geq0}^{2}\left\vert
\text{ }1\leq\xi_{1}<\xi_{2}\leq d+2\right.  \right\}
\]
is a bijective map with $\sharp\left\{  (\xi_{1},\xi_{2})\in\mathbb{Z}_{\geq
0}^{2}\left\vert \text{ }1\leq\xi_{1}<\xi_{2}\leq d+2\right.  \right\}
=\tbinom{d+2}{2}.$ On the other hand,%
\[%
\begin{array}
[c]{r}%
\sharp(\mathcal{E}_{2,d})=\sharp\left\{  (\alpha_{0},\alpha_{1},\alpha_{2}%
)\in\mathbb{Z}^{3}\left\vert \alpha_{0},\alpha_{1},\alpha_{2}\in
\lbrack0,d]\text{ and }\alpha_{0}+\alpha_{1}+\alpha_{2}\leq d\right.
\right\}  \medskip\\
-\sharp\left\{  (\alpha_{0},\alpha_{1},\alpha_{2})\in\mathbb{Z}^{3}\left\vert
\alpha_{0},\alpha_{1},\alpha_{2}\in\lbrack0,d-1]\text{ and }\alpha_{0}%
+\alpha_{1}+\alpha_{2}\leq d-1\right.  \right\}  \medskip\\
=\tbinom{d+3}{3}-\tbinom{d+2}{3}=\tbinom{d+2}{2}=\sharp(\text{Tr}_{d}%
\cap\mathbb{Z}^{2}).
\end{array}
\]
If $d\geq2,$ then using the homogeneous coordinates $[...:z_{(i,j)}%
:....]_{(i,j)\in\text{Tr}_{d}\cap\mathbb{Z}^{2}}$ we conclude (by Koelman's
Theorem \ref{KOELMANTHM}) that
\[
X_{\text{Tr}_{d}}\cong\mathbb{V}(I_{\mathcal{A}_{\text{Tr}_{d}}}%
)\subset\mathbb{P}_{\mathbb{C}}^{\delta_{\text{Tr}_{d}}},\text{ with }%
\delta_{\text{Tr}_{d}}=\tbinom{d+2}{2}-1,\text{ and}%
\]
{\small
\[
I_{\mathcal{A}_{\text{Tr}_{d}}}=\left\langle \left\{  z_{(i_{1},j_{1}%
)}z_{(i_{2},j_{2})}-z_{(i_{1}^{\prime},j_{1}^{\prime})}z_{(i_{2}^{\prime
},j_{2}^{\prime})}\left\vert
\begin{array}
[c]{c}%
(i_{1},j_{1}),(i_{2},j_{2}),(i_{1}^{\prime},j_{1}^{\prime}),(i_{2}^{\prime
},j_{2}^{\prime})\in\text{Tr}_{d}\cap\mathbb{Z}^{2},\medskip\\
\text{with }(i_{1},j_{1})+(i_{2},j_{2})=(i_{1}^{\prime},j_{1}^{\prime}%
)+(i_{2}^{\prime},j_{2}^{\prime})
\end{array}
\right.  \right\}  \right\rangle ,
\]
} i.e., that $X_{\text{Tr}_{d}}$ is isomorphic to the image of the so-called
$d$\textit{-uple} \textit{Veronese embedding}
\[
\nu_{2,d}:\mathbb{P}_{\mathbb{C}}^{2}\hookrightarrow\mathbb{P}_{\mathbb{C}%
}^{\delta_{\text{Tr}_{d}}},\ [\mathsf{X}_{0}:\mathsf{X}_{1}:\mathsf{X}%
_{2}]\longmapsto\lbrack...:\mathsf{X}^{\boldsymbol{\alpha}}%
:....]_{\boldsymbol{\alpha}\in\mathcal{E}_{2,d}}.
\]
where the monomials $\left\{  \mathsf{X}^{\boldsymbol{\alpha}}\left\vert
\boldsymbol{\alpha}\in\mathcal{E}_{2,d}\right.  \right\}  $ are arranged in a
prescribed manner (e.g., lexicographically). In fact, in this case,
\begin{align}
\beta_{\text{Tr}_{d}}  &  =\tbinom{\delta_{\text{Tr}_{d}}+2}{2}-\sharp
(2\text{Tr}_{d}\cap\mathbb{Z}^{2})\smallskip\nonumber\\
&  =\tfrac{1}{2}\tbinom{d+2}{2}(\tbinom{d+2}{2}+1)-(2d^{2}+3d+1)\nonumber\\
&  =\tfrac{(d+1)(d+2)}{4}(\tfrac{(d+1)(d+2)}{2}+1)-(2d^{2}+3d+1)=\tfrac{1}%
{8}d\left(  d+6\right)  \left(  d^{2}-1\right)  . \label{typosbeta}%
\end{align}

\end{examples}

\begin{note}
For a \texttt{Magma} code for the computation of a minimal generating system
of the ideal defining the projective toric surface associated to an
\textit{arbitrary} lattice polygon (and of much more, like Betti numbers
etc.), see \cite{CC2}. In the above mentioned particular case (in which we
deal only with \textit{quadrics}) it is enough (as we shall see in
\S \ref{ALGOR}) to collect all vectorial relations $(i_{1},j_{1})+(i_{2}%
,j_{2})=(i_{1}^{\prime},j_{1}^{\prime})+(i_{2}^{\prime},j_{2}^{\prime}),$ and
to determine a $\mathbb{C}$-linearly independent subset of the set of the
corresponding quadratic binomials $z_{(i_{1},j_{1})}z_{(i_{2},j_{2}%
)}-z_{(i_{1}^{\prime},j_{1}^{\prime})}z_{(i_{2}^{\prime},j_{2}^{\prime})}$ by
simply performing Gaussian elimination.
\end{note}

\section{The algorithm\label{ALGOR}}

\noindent{} An algorithm, implemented in \texttt{Python3} to compute a minimal
generating set for the ideal $I_{\mathcal{A}_{P}}$, given the vertex set
$\mathcal{V}(P)$ of the polygon $P$ is provided in the library
\texttt{toricIdeal.py}. The algorithm is provided by the routine
\texttt{minGenSet}, which receives a list of vertices and calls five
subroutines to compute a minimal generating set of $I_{\mathcal{A}_{P}}$.

\begin{lstlisting}[language=python]
import numpy
def minGenSet(p):
	intp=integerPoints(*vertToConst(p))
	(basis,genBin)=genBinom(intp)
	indepCol=findIndepCol(genBin)
	binom=findBinom(basis,genBin,indepCol)
	return(intp,binom)
\end{lstlisting}

More specifically, the first subroutine, \texttt{vertToConst}, produces a
complete system of facet-defining inequalities from the list $\mathcal{V}(P)$
and lower and upper bounds for the coordinates of the points of the polygon.
First, the facets are distinguished among the line segments connecting any two
vertices, using the fact that the polygon lies entirely in one of the two
closed half-planes bounded by their supporting lines and then a constraint is
created for each facet.

\begin{lstlisting}[language=Python]
def vertToConst(p):
	A=[]
	b=[]
	for i in range(len(p)-1):
		for j in range(i+1,len(p)):
			big=True
			small=True
			coef=numpy.array([(p[i,1]-p[j,1]),
				(p[j,0]-p[i,0])])
			cst=numpy.inner(coef,p[j])
			for point in p:
				diff=numpy.inner(coef,point)-cst
				if diff > 0:
					small=False
				elif diff <0:
					big=False
		if big :
			A.append(coef)
			b.append(cst)
		if small :
			A.append(-coef)
			b.append(-cst)
	A=numpy.array(A)
	b=numpy.array(b)
	x=[min(p[:,0]),max(p[:,0])]
	y=[min(p[:,1]),max(p[:,1])]
	return(A,b,x,y)
\end{lstlisting}

\noindent{} The second one, \texttt{integerPoints}, uses the constraints and by
brute force finds the lattice points of the polygon.

\begin{lstlisting}[language=python]
def integerPoints(A,b,x,y):
	intPoints=[]
	for x0 in range(x[0],x[1]+1):
		for y0 in range(y[0],y[1]+1):
			point=numpy.array([x0,y0])
			diff=numpy.dot(A,point)-b
			inInterior=True
			for coord in diff:
				if coord<0:
					inInterior=False
					break
			if inInterior:
				intPoints.append(point)
	intPoints=numpy.array(intPoints)
	return(intPoints)
\end{lstlisting}

\noindent{} The third one, \texttt{genBinom}, orders the basis elements
\[
B=\{z_{(i,j)}z_{(i^{\prime},j^{\prime})} : (i,j),(i^{\prime},j^{\prime})\in
P\cap\mathbb{Z} ^{2}\}
\]
of the $\mathbb{C}$-vector space $\text{HP}_{2}(\mathbb{P}_{\mathbb{C}%
}^{\delta_{P}})$ and finds a generating set for the ideal $I_{\mathcal{A}_{P}%
}$ by collecting all vectorial relations of the form
\[
(i_{1},j_{1})+(i_{2},j_{2})=(i_{1}^{\prime},j_{1}^{\prime})+(i_{2}^{\prime
},j_{2}^{\prime})
\]
where $(i_{1},j_{1}),(i_{2},j_{2}),(i_{1}^{\prime},j_{1}^{\prime}%
),(i_{2}^{\prime},j_{2}^{\prime})\in P\cap\mathbb{Z}^{2}$. This is returned as
a matrix containing the coefficients of the generating binomials w.r.t. the
ordered basis $B$.

\begin{lstlisting}[language=python]
def genBinom(intPoints):
	basis=[]
	for i in range(len(intPoints)):
		for j in range(i,len(intPoints)):
			bp=intPoints[i].tolist()+intPoints[j].tolist()
			basis.append(bp)
	basis=numpy.array(basis)
	genBin=[]
	for i in range(len(basis)):
		for j in range(i+1,len(basis)):
			if basis[i,0]+basis[i,2]==basis[j,0]+basis[j,2] and basis[i,1]+basis[i,3]==basis[j,1]+basis[j,3]:
				row=numpy.zeros(len(basis),
				dtype=numpy.int)
				row[i]+=1
				row[j]-=1
				genBin.append(row)
	genBin=numpy.transpose(numpy.array(genBin))
	return(basis, genBin)
\end{lstlisting}

\noindent{} The fourth one, \texttt{findIndepCol}, performs a Gauss elimination
on a matrix and finds a basis of its column space by collecting the non-zero
columns of the row echelon form of it.

\begin{lstlisting}[language=python]
from scipy.linalg import lu
def findIndepCol(A):
	U=lu(A,permute_l=True)[1]
	indepCol=[]
	for i in range(len(U)):
		for j in range(len(U[i])):
			if U[i,j]!=0:
				indepCol.append(j)
				break
	return(indepCol)
\end{lstlisting}

\noindent{} Finally, the fifth subroutine, \texttt{findBinom}, uses the matrix
given by \texttt{genBinom} and the list of $\mathbb{C}$-linearly independent
columns found by \texttt{findIndepCol} to produce a set of $\mathbb{C}%
$-linearly independent generating binomials of the ideal $I_{\mathcal{A}_{P}}$.

\begin{lstlisting}[language=python]
def findBinom(basis,genBin,indepCol):
	binom=[]
	for i in indepCol:
		j1=-1
		j2=-1
		for j in range(len(genBin)):
			if genBin[j,i]==1:
				j1=j
			elif genBin[j,i]==-1:
				j2=j
			if j1!=-1 and j2!=-1:
				break
			binomial="z_{{({},{})}}z_{{({},{})}}-z_{{({},{})}}z_{{({},{})}}".format(basis[j1][0],basis[j1][1],basis[j1][2],basis[j1][3],basis[j2][0],basis[j2][1],basis[j2][2],basis[j2][3])
	binom.append(binomial)
	return(binom)
\end{lstlisting}

\noindent{} The complexity of the \texttt{minGenSet} is polynomial of the
class $\mathcal{O}(m^{4}+n^{2})$, where $n$ is the number of vertices and $m$
an integer bounding absolutely the coordinates of the vertices.

\section{Applications\label{APPLIC}}

\noindent{} $\blacktriangleright$ \textbf{Veronese surfaces}. If
$P=\text{Tr}_{2}:=\text{conv}(\{(0,0),(2,0),(0,2)\})$ (with $d=2$ as in
\ref{Examplintro} (ii)), then the algorithm produces the following minimal
generating set of $I_{\mathcal{A}_{\text{Tr}_{2}}}$: {\small
\[%
\begin{array}
[c]{ll}%
z_{(0,0)}z_{(2,0)}-z_{(1,0)}^{2},\medskip & z_{(0,0)}z_{(0,2)}-z_{(0,1)}%
^{2},\\
z_{(0,0)}z_{(1,1)}-z_{(0,1)}z_{(1,0)},\medskip & z_{(0,1)}z_{(1,1)}%
-z_{(0,2)}z_{(1,0)},\\
z_{(0,2)}z_{(2,0)}-z_{(1,1)}^{2}, & z_{(0,1)}z_{(2,0)}-z_{(1,0)}z_{(1,1)}.
\end{array}
\]
}Analogously, if $P=\text{Tr}_{3}:=\text{conv}(\{(0,0),(3,0),(0,3)\})$ (with
$d=3$), then the $27$ quadrics

{\scriptsize
\[%
\begin{array}
[c]{lll}%
z_{(1,1)}z_{(3,0)}-z_{(2,0)}z_{(2,1)},\medskip & z_{(0,3)}z_{(2,1)}%
-z_{(1,2)}^{2}, & z_{(0,0)}z_{(1,2)}-z_{(0,2)}z_{(1,0)},\\
z_{(0,3)}z_{(3,0)}-z_{(1,2)}z_{(2,1)},\medskip & z_{(0,1)}z_{(3,0)}%
-z_{(1,0)}z_{(2,1)}, & z_{(0,1)}z_{(2,1)}-z_{(1,1)}^{2},\\
z_{(0,0)}z_{(2,0)}-z_{(1,0)}^{2},\medskip & z_{(0,2)}z_{(3,0)}-z_{(1,1)}%
z_{(2,1)}, & z_{(0,2)}z_{(2,1)}-z_{(0,3)}z_{(2,0)},\\
z_{(0,2)}z_{(2,1)}-z_{(1,1)}z_{(1,2)},\medskip & z_{(0,2)}z_{(3,0)}%
-z_{(1,2)}z_{(2,0)}, & z_{(0,1)}z_{(1,2)}-z_{(0,2)}z_{(1,1)},\\
z_{(0,0)}z_{(2,1)}-z_{(1,0)}z_{(1,1)},\medskip & z_{(0,0)}z_{(2,1)}%
-z_{(0,1)}z_{(2,0)}, & z_{(0,1)}z_{(2,1)}-z_{(0,2)}z_{(2,0)},\\
z_{(0,0)}z_{(3,0)}-z_{(1,0)}z_{(2,0)},\medskip & z_{(0,0)}z_{(0,2)}%
-z_{(0,1)}^{2}, & z_{(0,0)}z_{(1,1)}-z_{(0,1)}z_{(1,0)},\\
z_{(0,1)}z_{(0,3)}-z_{(0,2)}^{2},\medskip & z_{(1,0)}z_{(3,0)}-z_{(2,0)}%
^{2}, & z_{(0,1)}z_{(1,2)}-z_{(0,3)}z_{(1,0)},\\
z_{(0,2)}z_{(1,2)}-z_{(0,3)}z_{(1,1)},\medskip & z_{(1,2)}z_{(3,0)}%
-z_{(2,1)}^{2}, & z_{(0,1)}z_{(2,1)}-z_{(1,0)}z_{(1,2)},\\
z_{(0,0)}z_{(0,3)}-z_{(0,1)}z_{(0,2)},\medskip & z_{(0,1)}z_{(3,0)}%
-z_{(1,1)}z_{(2,0)}, & z_{(0,0)}z_{(1,2)}-z_{(0,1)}z_{(1,1)}%
\end{array}
\]
} generate minimally $I_{\mathcal{A}_{\text{Tr}_{3}}}$ (cf. (\ref{typosbeta}%
))$.\medskip$

\noindent{} $\blacktriangleright$ \textbf{Toric log del Pezzo surfaces}. These
are of the form $X_{P},$ where $P=\ell\mathring{Q}$ is the polar of an
LDP-polygon $Q\subset\mathbb{R}^{2}$ dilated by its index $\ell$. (An
\textit{LDP-polygon} $Q\subset\mathbb{R}^{2}$ is a convex polygon which
contains the origin in its interior, and its vertices belong to $\mathbb{Z}%
^{2}$ and are primitive. The \textit{index} of a polygon of this kind is
defined to be 
\[
\ell:=\min\{\left.  \kappa\in\mathbb{Z}_{>0}\right\vert
\mathcal{V}(\kappa\mathring{Q})\subset\mathbb{Z}^{2}\}.
\] 
Kasprzyk, Kreuzer \& Nill \cite[\S 6]{KKN} developed an algorithm by means of which one creates an
LDP-polygon, for given $\ell\geq2,$ by fixing a \textquotedblleft
special\textquotedblright\ edge and following a prescribed successive addition
of vertices, and produced in this way the long lists of \textit{all}
LDP-polygons for $\ell\leq17.$ An explicit study for each of these $15346$
LDP-polygons is available on the webpage \cite{Br-Kas}.)\medskip\ 

\noindent(i) Up to unimodular transformation the only \textit{reflexive
hexagon} (i.e., the only LDP-hexagon of index $1$) is%
\[
Q:=\text{ conv}\left(  \left\{
(0,1),(1,1),(1,0),(0,-1),(-1,-1),(-1,0)\right\}  \right)
\]
having
\[ 
\mathring{Q}:= \text{conv}\left(  \left\{
(1,0),(1,-1),(0,-1),(-1,0),(-1,1),(0,1)\right\}  \right) 
\]
as its polar, and
$X_{\mathring{Q}}\cong\mathbb{V}(I_{\mathcal{A}_{\mathring{Q}}})\subset
\mathbb{P}_{\mathbb{C}}^{6}$ with $I_{\mathcal{A}_{\mathring{Q}}}$ minimally
generated by the $9$ quadrics: {\small
\[%
\begin{array}
[c]{lll}%
z_{(-1,0)}z_{(1,-1)}-z_{(0,-1)}z_{(0,0)}, & z_{(-1,0)}z_{(1,0)}-z_{(0,0)}%
^{2}, & z_{(-1,0)}z_{(1,0)}-z_{(-1,1)}z_{(1,-1)},\\
z_{(-1,0)}z_{(0,0)}-z_{(-1,1)}z_{(0,-1)}, & z_{(-1,0)}z_{(0,1)}-z_{(-1,1)}%
z_{(0,0)}, & z_{(0,-1)}z_{(1,0)}-z_{(0,0)}z_{(1,-1)},\\
z_{(-1,0)}z_{(1,0)}-z_{(0,-1)}z_{(0,1)}, & z_{(-1,1)}z_{(1,0)}-z_{(0,0)}%
z_{(0,1)}, & z_{(0,0)}z_{(1,0)}-z_{(0,1)}z_{(1,-1)}.
\end{array}
\medskip
\]
} \noindent(ii) For the LDP-triangle $Q$ \textit{of index} $2$ with vertex
set
\[
\mathcal{V}(Q):=\{(0,1),(8,1),(-4,-1)\}
\]
we obtain 
\[\mathcal{V}(2\mathring{Q})=\{(1,-2),(0,-2),(-1,6)\} \text{(cf. Fig.
\ref{Fig.2})}
\] 
and $X_{2\mathring{Q}}\cong\mathbb{V}(I_{\mathcal{A}%
_{2\mathring{Q}}})\subset\mathbb{P}_{\mathbb{C}}^{6}$ with $I_{\mathcal{A}%
_{2\mathring{Q}}}$ minimally generated by the $7$ quadrics:{\small
\[%
\begin{array}
[c]{ll}%
z_{(-1,6)}z_{(1,-2)}-z_{(0,2)}^{2},\medskip & z_{(0,0)}z_{(0,2)}%
-z_{(0,1)}z_{(0,1)},\\
z_{(0,-1)}z_{(0,2)}-z_{(0,0)}z_{(0,1)},\medskip & z_{(0,-2)}z_{(0,2)}%
-z_{(0,0)}^{2},\\
z_{(0,-2)}z_{(0,2)}-z_{(0,-1)}z_{(0,1)},\medskip & z_{(0,-2)}z_{(0,0)}%
-z_{(0,-1)}^{2},\\
z_{(0,-2)}z_{(0,1)}-z_{(0,-1)}z_{(0,0)}. &
\end{array}
\medskip
\]
}\begin{figure}[th]
\includegraphics[height=7cm, width=11.5cm]{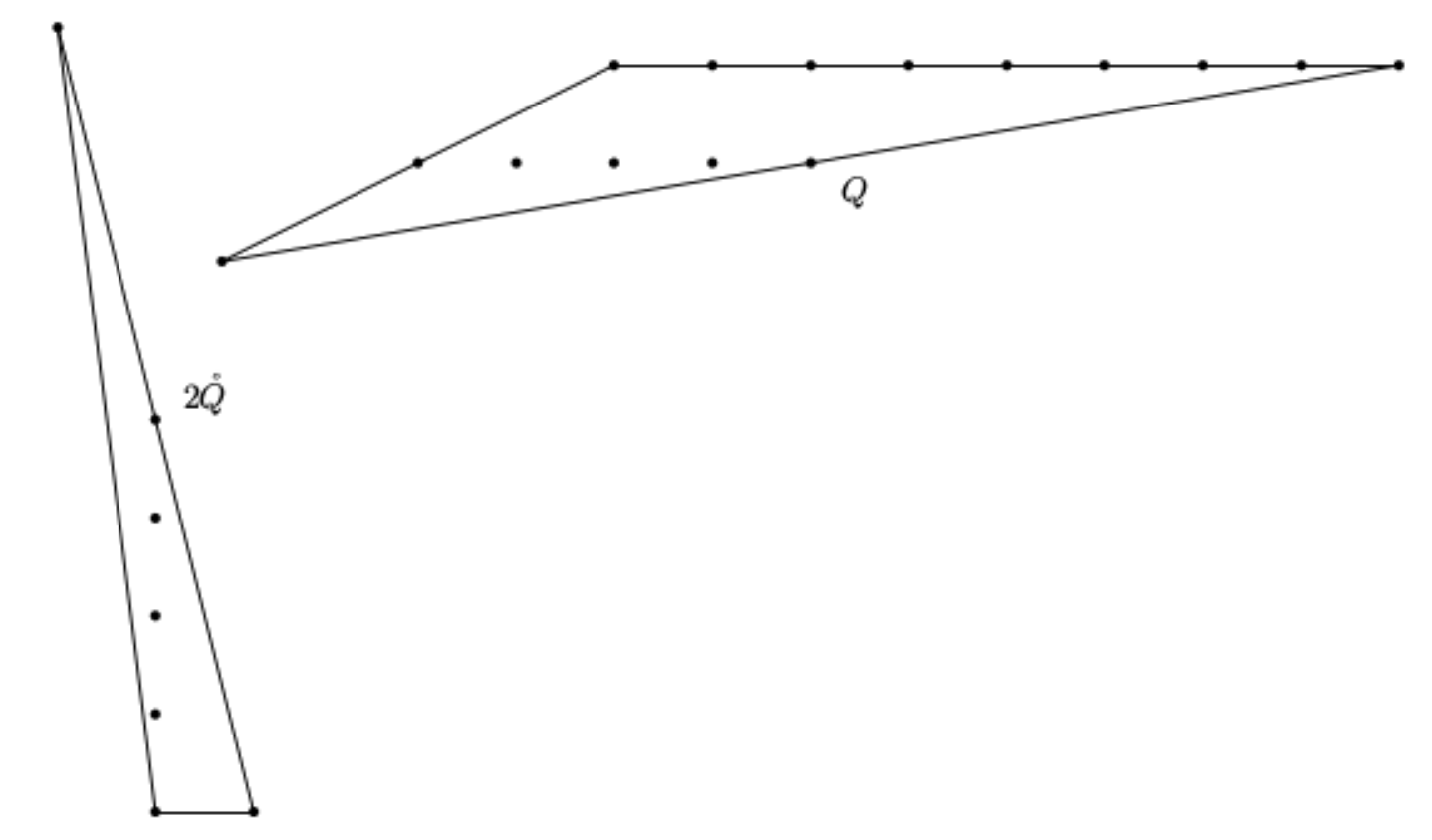} \caption{{}}%
\label{Fig.2}%
\end{figure}

\noindent(iii) For the LDP-pentagon $Q$ \textit{of index} $3$ with vertex set
\[
\mathcal{V}(Q):=\{(0,1),(1,1),(1,0),(-2,-1),(-3,-1)\}
\]
(which is unimodularly equivalent to the pentagon \textquotedblleft%
$Q_{2}^{[3]}$\textquotedblright\ of \cite{Dais}) we obtain
\[
\mathcal{V}(3\mathring{Q})=\{(2,-3),(0,-3),(-3,0),(-3,-9),(0,3)\},
\]
and $X_{3\mathring{Q}}\cong\mathbb{V}(I_{\mathcal{A}_{3\mathring{Q}}}%
)\subset\mathbb{P}_{\mathbb{C}}^{38}$ with $I_{\mathcal{A}_{3\mathring{Q}}}$
minimally generated by a set of $646$ quadrics! \medskip

\noindent(iv) For the LDP-quadrilateral $Q$ \textit{of index} $4$ with vertex
set
\[
\mathcal{V}(Q):=\{(-1,2),(3,2),(-1,-1),(-3,-2)\}
\]
we obtain $
\mathcal{V}(4\mathring{Q})=\{(2,-1),(0,-2),(-12,16),(-4,8)\},$
and $X_{4\mathring{Q}}\cong\mathbb{V}(I_{\mathcal{A}_{4\mathring{Q}}}%
)\subset\mathbb{P}_{\mathbb{C}}^{45}$ with $I_{\mathcal{A}_{4\mathring{Q}}}$
minimally generated by a set of $918$ quadrics!\medskip

\noindent(v) Finally, the LDP-triangle $Q$ \textit{of index} $5$ with vertex
set%
\[
\mathcal{V}(Q):=\{(0,1),(15,1),(-15,-2)\}
\]
we obtain
$\mathcal{V}(5\mathring{Q})=\{(1,-5),(0,-5),(-1,10)\},$
and $X_{5\mathring{Q}}\cong\mathbb{V}(I_{\mathcal{A}_{5\mathring{Q}}}%
)\subset\mathbb{P}_{\mathbb{C}}^{9}$ with $I_{\mathcal{A}_{5\mathring{Q}}}$
minimally generated by the following $21$ quadrics:{\small
\[%
\begin{array}
[c]{lll}%
z_{(0,-5)}z_{(0,-1)}-z_{(0,-3)}^{2},\medskip & z_{(0,-5)}z_{(0,1)}%
-z_{(0,-4)}z_{(0,0)}, & z_{(0,-5)}z_{(0,0)}-z_{(0,-3)}z_{(0,-2)},\\
z_{(0,-5)}z_{(0,2)}-z_{(0,-4)}z_{(0,1)},\medskip & z_{(0,-5)}z_{(0,2)}%
-z_{(0,-3)}z_{(0,0)}, & z_{(0,-3)}z_{(0,2)}-z_{(0,-2)}z_{(0,1)},\\
z_{(0,-4)}z_{(0,2)}-z_{(0,-2)}z_{(0,0)},\medskip & z_{(0,-2)}z_{(0,2)}%
-z_{(0,-1)}z_{(0,1)}, & z_{(0,-5)}z_{(0,-1)}-z_{(0,-4)}z_{(0,-2)},\\
z_{(0,-5)}z_{(0,1)}-z_{(0,-3)}z_{(0,-1)},\medskip & z_{(0,-5)}z_{(0,-3)}%
-z_{(0,-4)}^{2}, & z_{(0,-4)}z_{(0,2)}-z_{(0,-1)}^{2},\\
z_{(0,-5)}z_{(0,0)}-z_{(0,-4)}z_{(0,-1)},\medskip & z_{(0,-5)}z_{(0,1)}%
-z_{(0,-2)}^{2}, & z_{(0,-1)}z_{(0,2)}-z_{(0,0)}z_{(0,1)},\\
z_{(0,-2)}z_{(0,2)}-z_{(0,0)}^{2},\medskip & z_{(0,0)}z_{(0,2)}-z_{(0,1)}%
^{2}, & z_{(0,-5)}z_{(0,2)}-z_{(0,-2)}z_{(0,-1)},\\
z_{(0,-5)}z_{(0,-2)}-z_{(0,-4)}z_{(0,-3)}, & z_{(0,-3)}z_{(0,2)}%
-z_{(0,-1)}z_{(0,0)}, & z_{(0,-4)}z_{(0,2)}-z_{(0,-3)}z_{(0,1)}.
\end{array}
\medskip
\]
}

\end{document}